\documentclass[11pt]{article}

\date{}
\title{Graph functionality}
\author{Bogdan Alecu\thanks{Mathematics Institute, University of Warwick, Coventry, CV4 7AL, UK. B.Alecu@warwick.ac.uk} \and
Aistis Atminas\thanks{Department of Mathematics, London School of Economics, London WC2A 2AE, UK, A.Atminas@lse.ac.uk} \and 
Vadim Lozin\thanks{Mathematics Institute, University of Warwick, Coventry, CV4 7AL, UK. V.Lozin@warwick.ac.uk}
}

\usepackage{caption}
\usepackage{tabularx}
\usepackage{amsthm} 
\usepackage{amsmath} 
\usepackage{amsfonts}
\usepackage[cp1251]{inputenc}
\usepackage{amsmath, amsthm, amssymb}
\usepackage{colordvi}
\usepackage{graphicx}
\usepackage{tikz}
\usepackage{subcaption}
\usetikzlibrary{decorations.pathreplacing,arrows,automata}

\tikzstyle{w_vertex}=[circle,fill=black!100,text=white,inner sep=0.4mm,draw]
\tikzstyle{vertex}=[circle,fill=black!100,text=white,inner sep=0.8mm]
\tikzstyle{point}=[circle,fill=black,inner sep=0.1mm]

\begin{document}
\maketitle

\newtheorem{theorem}{Theorem}
\newtheorem{obs}{Observation}
\newtheorem{lemma}{Lemma}
\newtheorem{proposition}{Proposition}
\newtheorem{prop}{Proposition}
\newtheorem{cor}{Corollary}
\newtheorem{definition}{Definition}
\newtheorem{remark}{Remark}
\newtheorem{conjecture}{Conjecture}
\newtheorem{problem}{Open problem}

\def\N{\mathbb{N}}
\def\t{\sim}
\def\nt{\nsim}
\def\1{n+1}
\def\2{n+2}

\begin{abstract}
Let $G=(V,E)$ be a graph and $A$ its adjacency matrix. We say that a vertex $y \in V$ is a function of vertices $x_1, \ldots, x_k \in V$ 
if there exists a Boolean function $f$ of $k$ variables such that for any vertex $z \in V - \{y, x_1, \ldots, x_k\}$,
$A(y,z)=f(A(x_1,z),\ldots,A(x_k,z))$. The functionality $fun(y)$ of vertex $y$ is the minimum $k$ such that $y$ is a function of $k$ vertices. 
The functionality $fun(G)$ of the graph $G$ is $\max\limits_H\min\limits_{y\in V(H)}fun(y)$, where the maximum is taken over all induced subgraphs $H$ of $G$.  
In the present paper, we show that functionality generalizes simultaneously several other graph parameters, such as degeneracy or clique-width, by proving that 
bounded degeneracy or bounded clique-width imply bounded functionality. Moreover, we show that this generalization is proper by revealing classes of graphs of
unbounded degeneracy and clique-width, where functionality is bounded by a constant. This includes permutation graphs, unit interval graphs and line graphs. 
We also observe that bounded functionality implies bounded VC-dimension, i.e. graphs of bounded VC-dimension extend graphs of bounded functionality, and this extension is also proper. 
\end{abstract}



\section{Introduction}

Let $G=(V,E)$ be a simple graph, i.e. undirected graph without loops and multiple edges.  
We denote by $A=A_G$ the adjacency matrix of $G$ and by $A(x,y)$ the element of this matrix
corresponding to vertices $x,y\in V$, i.e. $A(x,y)=1$ if $x$ and $y$ are adjacent, and $A(x,y)=0$ otherwise. 

We say that a vertex $y \in V$ is a function of vertices $x_1, \ldots, x_k \in V$ 
if there exists a Boolean function $f$ of $k$ variables such that for any vertex $z \in V - \{y, x_1, \ldots, x_k\}$,
$$A(y,z)=f(A(x_1,z),\ldots,A(x_k,z)).$$ 
The functionality $fun(y)$ of vertex $y$ is the minimum $k$ such that $y$ is a function of $k$ vertices. 
In particular, the functionality of an isolated vertex is $0$, and the same is true for a dominating (also known as universal) vertex, i.e. a vertex adjacent to 
all the other vertices of the graph. More generally, the functionality of a vertex $y$ does not exceed the number of its neighbours (the degree of $y$) and the number of its non-neighbours. 
One more simple example of functional vertices is given by twins, i.e. vertices $x$ and $y$ that have the same set of neighbours different from $x$ and $y$. 
Twins are functions of each other and their functionality is (at most) 1. 
    
\medskip
The functionality $fun(G)$ of $G$ is $$\max\limits_H\min\limits_{y\in V(H)}fun(y),$$ where the maximum is taken over all induced subgraphs $H$ of $G$.  
Similarly to many other graph parameters, this notion becomes valuable when its value is small, i.e. is bounded by a constant independent  of the size of the graph. 
This is important, in particular, for coding of graphs, i.e. representing them by words in a finite alphabet, which is needed for representing graphs in computer memory.
Indeed, if a vertex $y$ is a function of only constantly many vertices, then to describe the neighbourhood of $y$ we need $O(\log_2 n)$ bits, regardless of how large the neighbourhood (or non-neighbourhood)
of $y$ is.

In the present paper, we explore the relationship between graph functionality and other graph parameters.
From the above discussion, it follows that graphs of bounded functionality extend graphs of bounded vertex degree.
More generally, they extend graphs of bounded degeneracy, where the {\it degeneracy} of $G$ is the minimum $k$ such that every induced subgraph of $G$ has a vertex of degree at most $k$. 
A notion related to degeneracy is that of {\it arboricity}, which is the minimum number of forests into which the edges of $G$ can be partitioned.
The degeneracy of $G$ is always between the arboricity and twice the arboricity of $G$ and hence graphs of bounded functionality extend graphs of bounded arboricity too.

One more important graph parameter is {\it clique-width}. Many algorithmic problems that are generally NP-hard become polynomial-time solvable when restricted 
to graphs of bounded clique-width \cite{CMR00}. Clique-width is a relatively new notion and it generalizes another important graph parameter, {\ tree-width}, 
studied in the literature for decades. Clique-width is stronger than tree-width in the sense that graphs of bounded tree-width have bounded clique-width. 
In Section~\ref{sec:bounded}, we show that functionality is stronger than clique-width by proving that graphs of bounded clique-width have bounded functionality.
Moreover, in the same section we reveal three classes of graphs, where functionality is bounded but clique-width and degeneracy are not. These are permutation graphs, unit interval graphs and line graphs.  

In \cite{implicit}, it was shown that any class of graphs of bounded functionality contains  $2^{O(n\log_2 n)}$ labelled graphs with $n$ vertices. 
Therefore, functionality is unbounded in any larger class of graphs. In particular, it is unbounded in the classes of bipartite graphs, co-bipartite graphs and split graphs. 
In \cite{Lozin}, it was show that these are the only three minimal hereditary classes of graphs of unbounded VC-dimension.
Therefore, graphs of bounded VC-dimension extend graphs of bounded functionality. Moreover, this extension is proper, as we show in Section~\ref{sec:unbounded}.
Section~\ref{sec:conclusion} concludes the paper with a number of other open problems. In the rest of the present section we introduce basic terminology and notation used in the paper. 

For a simple graph $G$, we denote by $V(G)$ and $E(G)$ the vertex set and the edge set of $G$, respectively. 
The neighbourhood $N(v)$ of a vertex $v\in V(G)$ is the set of vertices of $G$ adjacent to $v$ and the degree of $v$ is $|N(v)|$.
A vertex of degree 0 is called {\it isolated}. 
The closed neighbourhood of $v$ is $N[v]=\{v\}\cup N(v)$.

A {\it clique} in a graph $G$ is a set of pairwise adjacent vertices, and an {\it independent set} is a 
set of pairwise non-adjacent vertices. The {\it girth} of $G$ is the length of a shortest cycle in $G$. 
A chordless cycle of length $n$ is denoted $C_n$. 
 
A {\it forest} is a graph without cycles and a {\it tree} is a connected graph without cycles. A graph $G$ is {\it bipartite} if
$V(G)$ can be partitioned into two independent sets,  {\it co-bipartite} if
$V(G)$ can be partitioned into two cliques, and {\it split} if $V(G)$ can be partitioned into a clique and an independent set.  

A graph $H$ is an {\it induced subgraph} of a graph $G$ if $H$ can be obtained from $G$ by vertex deletions. 
A class $X$ of graphs is {\it hereditary} if it is closed under taking induced subgraphs, or equivalently, if it is closed under deletion of vertices from graphs in the class.
A class $X$ is {\it monotone} if it is closed under vertex deletions and edge deletions, and $X$ is {\it minor-closed} if it is closed under vertex deletions, edge deletions and edge contractions.
Clearly, every minor-closed class is monotone and every monotone class is hereditary.

It is well-known (and not difficult to see) that a class $X$ of graphs is hereditary if and only if it can be described by means of minimal forbidden induced subgraphs,
i.e. vertex-minimal graphs that do not belong to $X$. If $M$ is the set of minimal forbidden induced subgraphs for $X$, then we say that graphs in $X$ are $M$-free,
and if $M$ is finite, we say that $X$ is {\it finitely defined}.  

By $A\otimes B$ we denote the symmetric difference of two sets, i.e. $A\otimes B=(A-B)\cup (B-A)$. When taking the symmetric difference of vertex neighbourhoods, 
we will always exclude the two vertices themselves; for brevity, we will write $N(u) \otimes N(v)$ to mean the set of vertices different from $u$ and $v$ and adjacent to 
exactly one of $u$ and $v$.


\section{Graphs of small functionality}
\label{sec:bounded}

As we mentioned in the introduction, functionality is bounded for graphs of bounded degree or degeneracy, which is easy to see. 
This includes, in particular, all proper minor-closed classes of graphs. The family of monotone classes is larger and not all
classes in this family are of bounded functionality. In this paper, we present a dichotomy  with respect to bounded/unbounded functionality
for monotone classes defined by finitely many forbidden induced subgraphs. The first part of the dichotomy describes monotone classes of bounded 
functionality without restriction to finitely defined classes and is presented in Section~\ref{sec:monotone}. 
The other part of the dichotomy applies to finitely defined monotone classes only and is presented in Section~\ref{sec:unbounded}.

\subsection{Monotone classes of bounded functionality}
\label{sec:monotone}

\begin{theorem}\label{thm:monotone-bounded}
If $X$ is a monotone class that does not contain at least one forest, then graphs in $X$ have bounded functionality. 
\end{theorem}

\begin{proof}
Let $F$ be a forest that does not belong to $X$, and let $k$ be the number of vertices in $F$.  Assume that $X$ has a graph $G$ every vertex of which has degree at least $k$.
Then $G$ contains every tree $T$ with at most $k+1$ vertices as a (not necessarily induced) subgraph, which can be easily shown by induction on the number of vertices in $T$.   
But then $G$ contains $F$ as a subgraph, which contradicts the assumption that $X$ is a monotone class that does not contain $F$. This shows that every graph in $X$ 
contains a vertex of degree at most $k-1$. Since $X$ is hereditary, we conclude that the degeneracy of graphs in $X$ is at most $k-1$. Therefore, graphs in $X$ have bounded functionality.   
\end{proof}

\subsection{Graphs of bounded clique-width}

The notion of clique-width of a graph was introduced in \cite{CER93}. The clique-width of a graph $G$ is denoted ${\rm cwd}(G)$
and is defined as the minimum number of labels needed to construct $G$ by means of the following four graph operations: 
\begin{itemize}
\item
creation of a new vertex $v$ with 
label $i$ (denoted $i(v)$), 
\item disjoint union of two labelled graphs $G$ and 
$H$ (denoted $G\oplus H$), 
\item connecting vertices with specified labels $i$ 
and $j$ (denoted $\eta_{i,j}$) and 
\item renaming label $i$ to label $j$ 
(denoted $\rho_{i\to j}$). 
\end{itemize}

Every graph can be defined by an algebraic expression using the four operations above. 
This expression is called a $k$-expression if it uses $k$ different labels.
For instance, the cycle $C_5$ on vertices $a,b,c,d,e$ 
(listed along the cycle) can be defined by the following 4-expression:
$$
\eta_{4,1}(\eta_{4,3}(4(e)\oplus\rho_{4\to 3}(\rho_{3\to 2}(\eta_{4,3}(4(d)\oplus\eta_{3,2}(3(c)\oplus\eta_{2,1}(2(b)\oplus 1(a)))))))).
$$

Alternatively, any algebraic expression defining $G$ can be represented as a rooted tree, 
whose leaves correspond to the operations of vertex creation, the internal nodes correspond 
to the $\oplus$-operations, and the root is associated with $G$. The operations $\eta$ and 
$\rho$ are assigned to the respective edges of the tree. Figure~\ref{fig:tree} shows the tree 
representing the above expression defining a $C_5$. 

\bigskip
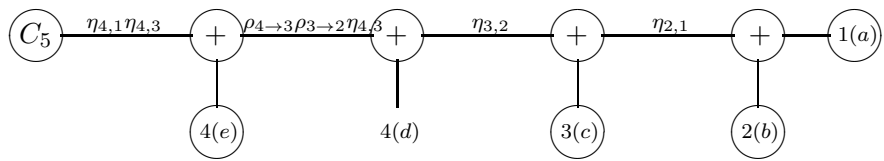
\begin{figure}[ht]
\begin{center}
\setlength{\unitlength}{0.32mm}
\begin{picture}(370,50)
\put(15,50){\circle{20}}
\put(90,50){\circle{20}}
\put(165,50){\circle{20}}
\put(240,50){\circle{20}}
\put(315,50){\circle{20}}
\put(355,50){\circle{20}}
\put(90,10){\circle{20}}
\put(240,10){\circle{20}}
\put(315,10){\circle{20}}
\put(25,50){\line(1,0){55}}
\put(100,50){\line(1,0){55}}
\put(175,50){\line(1,0){55}}
\put(250,50){\line(1,0){55}}
\put(325,50){\line(1,0){20}}
\put(90,40){\line(0,-1){20}}
\put(165,40){\line(0,-1){20}}
\put(240,40){\line(0,-1){20}}
\put(315,40){\line(0,-1){20}}
\put(85,47){+}
\put(160,47){+}
\put(235,47){+}
\put(311,47){+}
\put(8,47){$C_5$}
\put(83,10){$_{4(e)}$}
\put(158,10){$_{4(d)}$}
\put(233,10){$_{3(c)}$}
\put(308,10){$_{2(b)}$}
\put(348,50){$_{1(a)}$}
\put(101,55){$_{\rho_{4\to 3}\rho_{3\to 2}\eta_{4,3}}$}
\put(36,55){$_{\eta_{4,1}\eta_{4,3}}$}
\put(196,55){$_{\eta_{3,2}}$}
\put(271,55){$_{\eta_{2,1}}$}
\end{picture}
\end{center}
\caption{The tree representing the expression defining a $C_5$}
\label{fig:tree}
\end{figure}

Among various examples of graphs of bounded clique-width we mention distance-hereditary graphs. 
These are graphs of clique-width at most 3 \cite{perfect-cw}. Every graph in this class can be constructed 
from a single vertex by successively adding either a pendant vertex  or a twin (true or false) \cite{DH}.
From this characterization we immediately conclude that the functionality of distance-hereditary graphs is at most one. 
More generally, in the next theorem we show that functionality is bounded for all classes of graphs of bounded clique-width.

\begin{theorem}
For any graph $G$, $fun(G)\le 2{\rm cwd}(G)-1$.
\label{cwd}
\end{theorem}

\begin{proof}
Let $G$ be a graph of clique-width $k$ and let $T$ be a tree corresponding to a $k$-expression that describes $G$.
Consider a node $v$ of the tree such that the tree rooted at $v$ has more than $k$ leaves, and no children of $v$ have this property (if no such $v$ exists, we are done, since $G$ has at most $k$ vertices). 
Denote the children of $v$ by $u_1,\ldots, u_t$ and let $i$ be the minimum index such that $u_1, \ldots, u_i$ have a combined total of more than $k$ leaves amongst their descendants. 
Consider the subtree $T'$ of $T$ consisting of $v$ and the union of the trees rooted at $u_1, \ldots, u_i$. This subtree has more than $k$ leaves and therefore 
at least two of them, say $x$ and $y$, have the same label at node $v$. On the other hand, $T'$ has at most $2k$ leaves by the choice of $i$.  
Therefore, the symmetric difference $N(x)\otimes N(y)$ contains at most $2k-2$ vertices, since $x$ and $y$ are not distinguished outside of the tree rooted at $v$. 
As a result, the functionality of both $x$ and $y$ is at most $2k-1$ ($x$ is a function of $\{y\}\cup (N(x)\otimes N(y))$ and similarly $y$ is a function of $\{x\}\cup (N(x)\otimes N(y))$).  

It is known (see e.g. \cite{CO00}) that the clique-width of an induced subgraph of $G$ cannot exceed the clique-width of $G$. Therefore, every induced subgraph of $G$ 
has a vertex of functionality at most $2k-1$. Thus, the functionality of $G$ is at most $2k - 1$.
\end{proof}

This result shows that the family of graph classes of bounded functionality extends the family of graph classes of bounded clique-width. Moreover, this extension is proper,
because clique-width is known to be unbounded for square grids. This example, however, is not very interesting in the sense that  square grids have bounded vertex degree and hence bounded functionality. 
In the next three sections, we reveal several classes of graphs of bounded functionality, where neither clique-width nor degeneracy is bounded.

\subsection{Unit interval graphs}

A unit interval graph is the intersection graph of intervals of the same length on the real line. 
The class of unit interval graphs is one of the minimal hereditary classes of unbounded clique-width \cite{Lozin-minimal}. 
Also, degeneracy is unbounded in this class, since it contains cliques of arbitrarily large size. 
Our next result shows that functionality is bounded for unit interval graphs.  

\begin{theorem}\label{thm:unit-interval}
	The functionality of  unit interval graphs is at most 2.
\end{theorem}
\begin{proof}
Let $G$ be a unit interval graph with $n$ vertices and assume without loss of generality that $G$ has no isolated vertices (since any such vertex has functionality 0). 
Take a unit interval representation for $G = (V, E)$ with the interval endpoints all distinct. 
We label the vertices $v_1,\ldots,v_n$  in the order in which they appear on the real line (from left to right), and denote 
the endpoints of interval  $I_i$ corresponding to vertex $v_i$ by $a_i < b_i$.  
We will bound $$S=\sum\limits_{i=1}^{n - 1} |N(v_i) \otimes N(v_{i+1})|.$$
	
Note that any neighbour of $v_i$ which is not a neighbour of $v_{i+1}$ needs to have its right endpoint between $a_i$ and $a_{i+1}$. 
Similarly, any neighbour of $v_{i+1}$ but not of $v_i$ needs to have its left endpoint between $b_i$ and $b_{i+1}$. 
In other words, $|N(v_i) \otimes N(v_{i+1})|$ is bounded above by the number of endpoints in $(a_i, a_{i+1}) \cup (b_i, b_{i+1})$ 
(we say bounded above and not equal, since it might happen that $b_i$ lies between $a_i$ and $a_{i+1}$, without contributing to the symmetric difference).
	
The key is now to note that any endpoint can be counted at most once in the whole sum $S$, since all $(a_i, a_{i+1})$ are disjoint (and the same applies to the $(b_i, b_{i+1})$), 
and the $a$'s can only appear between $b$'s (and vice-versa). In fact, $a_1$ and $b_n$ are never counted in $S$, and if $a_2$ is between $b_1$ and $b_2$, then $v_1$ must be isolated, so $a_2$ is not counted either. 
The sum is thus at most $2n - 3$. Since it has $n - 1$ terms, one of the terms, say $|N(v_t) \otimes N(v_{t+1})|$, must be at most 1. Therefore, the functionality of both $v_t$ and $v_{t+1}$ is at most $2$.

We have proved that each unit interval graph has a vertex of functionality at most $2$. Since this class is hereditary, we conclude that the functionality of any unit interval graph is at most~$2$.
\end{proof}

\subsection{Permutation graphs}
Let $\pi$ be a permutation of the elements in $\{1,2,\ldots,n\}$. The permutation graph of $\pi$ is a graph with vertex set $\{1,2,\ldots,n\}$ in which two vertices 
$i$ and $j$ are adjacent if and only if $(i-j)(\pi(i)-\pi(j))<0$. Clique-width is known to be unbounded in the class of permutation graphs \cite{perfect-cw}, and so is degeneracy. 
In Section~\ref{sec:perm-bound}, we show that functionality is bounded by a constant in this class. A similar result for unit interval graphs and graphs of bounded clique-width 
was proved by finding a pair of vertices with bounded symmetric difference of their neighbourhoods. This is not the case for permutation graphs, as we show in Section~\ref{sec:perm-sd}.
This result is of independent interest, because in conjunction with Theorem~\ref{cwd} it provides an alternative proof of the fact that clique-width is unbounded in the class of permutation
graphs.

\subsubsection{Functionality is bounded for permutation graphs}
\label{sec:perm-bound}
For the purpose of this section, we associate a permutation $\pi$ with its plot, i.e. 	the set of points $(i, \pi(i))$ in the plane. 
We label those points by $\pi(i)$ and define the {\em geometric neighbourhood} of a point $k$ to be the union of two regions in the plane: 
the one above and to its left, and the one below and to its right. Then it is not difficult to see that the set of points of the permutation 
lying in the geometric neighbourhood of $k$ is precisely the set of neighbours of vertex $k$ in the permutation graph of $\pi$.

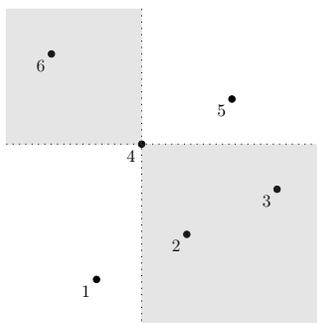
\begin{figure}[ht]
\begin{center}
\begin{tikzpicture}[scale=0.6, transform shape]
\filldraw 
(1,6) circle (2pt) node[below left] {6}
(2,1) circle (2pt) node[below left] {1}
(3,4) circle (2pt) node[below left] {4}
(4,2) circle (2pt) node[below left] {2}
(5,5) circle (2pt) node[below left] {5}
(6,3) circle (2pt) node[below left] {3};`

\draw[dotted] (3, 7) -- (3, 0);
\draw[dotted] (0, 4) -- (7, 4);

\fill[gray, fill opacity = 0.2] (3,4) -- (0,4) -- (0,7) -- (3,7) -- cycle;
\fill[gray, fill opacity = 0.2] (3,4) -- (7,4) -- (7,0) -- (3,0) -- cycle;
\end{tikzpicture}
\end{center}
\caption{Geometric representation of $\pi = 614253$, with the neighbourhood of 4 shaded}
\label{examplepermutation}	
\end{figure}

\begin{theorem}
	The functionality of permutation graphs is at most 8.
\end{theorem}
\begin{proof}
Since the class of permutation graphs is hereditary, it suffices to show that every permutation graph contains a vertex of functionality at most 8. 
Let $G$ be a permutation graph corresponding to a permutation $\pi$. 
The proof will be given in two steps: first, we show that if there is a vertex with a specific property in $G$, 
then this vertex is a function of 4 other vertices. Second, we show how to find vertices that are ``close enough'' to having that property.
	
	\medskip
{\it Step} 1: Consider the plot of $\pi$. 
Among any 3 horizontally consecutive points, one is vertically between the two others. 
We call such a point {\em vertical middle} (in the permutation from Figure~\ref{examplepermutation}, the vertical middle points are 4, 2 and 3). 
Similarly, among any 3 vertically consecutive points, one is horizontally between the two others, and we call this point {\em horizontal middle} 
(in Figure~\ref{examplepermutation}, the horizontally middle points are 2, 5 and 4). 
	
Now let us suppose that $\pi$ has a point $x$ that is simultaneously a horizontal and a vertical middle point. 
Then $x$ is part of a triple $x$, $b$, $t$ (not necessarily in that order) of horizontally consecutive points, 
where $b$ is the bottom point (the lowest in the triple) and $t$ is the top point (the highest in the triple).
Also, $x$ is part  of a triple $x$, $l$, $r$ (not necessarily in that order) of vertically consecutive points, 
where $l$ is the leftmost and $r$ is the rightmost point in the triple (see Figure~\ref{middlepoint} for an illustration).

In general, $x$ can be at any of the 9 intersection points of pairs of 3 consecutive vertical and horizontal lines, i.e. $x$ is somewhere in $X$ (see Figure~\ref{middlepointareas}). 
We also have $l \in L$, $r \in R$, $t \in T$ and $b \in B$ for the surrounding points (see Figure~\ref{middlepointareas}). The important thing to note is that, 
since the points are consecutive, those are the {\em only} points of the permutation lying in the shaded area $X \cup L \cup R \cup T \cup B$.
Any point different from $x,l,r,t,b$ lies in one of $Q_1$, $Q_2,Q_3$ or $Q_4$.

\begin{figure}[ht]
\centering
\begin{subfigure}[t]{0.49\linewidth}
\centering
\begin{tikzpicture}[scale=0.6, transform shape]
	
	\draw[dotted] (3.5,0) -- (3.5,10);
	\draw[dotted] (4,0) -- (4,10);
	\draw[dotted] (4.5,0) -- (4.5,10);
	\draw[dotted] (0,5.5) -- (10,5.5);
	\draw[dotted] (0,6) -- (10,6);
	\draw[dotted] (0,6.5) -- (10,6.5);
	
	\filldraw (3.5,5.5) circle (2pt) node[below left] {$x$};
	\filldraw (1.5,6.5) circle (2pt) node[below left] {$l$};
	\filldraw (7,6) circle (2pt) node[below left] {$r$};
	\filldraw (4,2) circle (2pt) node[below left] {$b$};
	\filldraw (4.5,8) circle (2pt) node[below left] {$t$};	
	\fill[gray, fill opacity = 0.2] (4,6) -- (4,10) -- (0,10) -- (0,6) -- cycle;
	\fill[gray, fill opacity = 0.2] (4.5,6.5) -- (10,6.5) -- (10,0) -- (4.5,0) -- cycle;	
\end{tikzpicture}
\captionsetup{justification=centering}
\caption{The geometric neighbourhood corresponding to $(N(r) \cap N(b)) \cup (N(l) \cap N(t))$}
\label{middlepoint}	
\end{subfigure}	
\begin{subfigure}[t]{0.49\linewidth}
\centering
\begin{tikzpicture}[scale=0.6, transform shape]

\fill[lightgray, fill opacity = 0.1] (0,6.5) -- (3.5, 6.5) -- (3.5,10) -- (4.5,10) -- (4.5,6.5) -- (10,6.5) -- (10, 5.5) -- (4.5,5.5) -- (4.5,0) -- (3.5,0) -- (3.5, 5.5) -- (0,5.5) -- cycle;

\draw[dotted] (3.5,0) -- (3.5,10);
\draw[dotted] (4.5,0) -- (4.5,10);
\draw[dotted] (0,5.5) -- (10,5.5);
\draw[dotted] (0,6.5) -- (10,6.5);



\draw (4,6) node{$X$};
\draw (1.75,6) node{$L$};
\draw (7.25,6) node{$R$};
\draw (4,8.25) node{$T$};
\draw (4,2.75) node{$B$};
\draw (7.25,8.25) node{$Q_1$};
\draw (1.75,8.25) node{$Q_2$};
\draw (1.75,2.75) node{$Q_3$};
\draw (7.25,2.75) node{$Q_4$};

\fill[fill = gray, fill opacity = 0.2] (0,6.5) -- (3.5, 6.5) -- (3.5,10) -- (4.5,10) -- (4.5,6.5) -- (10,6.5) -- (10, 5.5) -- (4.5,5.5) -- (4.5,0) -- (3.5,0) -- (3.5, 5.5) -- (0,5.5) -- cycle;
\end{tikzpicture}
\captionsetup{justification=centering}
\caption{Partition of the plot}
\label{middlepointareas}	
\end{subfigure}
\caption{A middle point $x$ and its four surrounding points}
\end{figure}
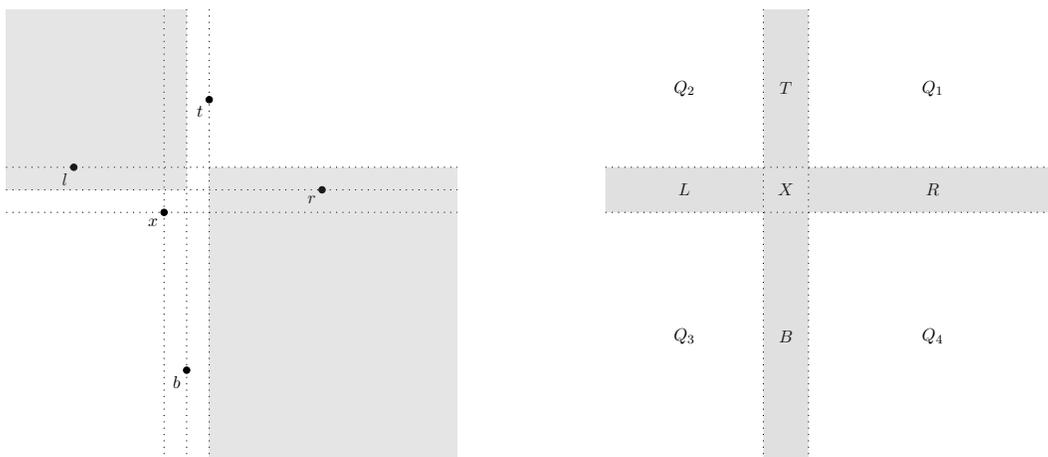

It is not difficult to see that the geometric neighbourhood corresponding to $(N(r) \cap N(b)) \cup (N(l) \cap N(t))$ (see Figure~\ref{middlepoint}) will always contain $Q_2$ and $Q_4$, 
and will never intersect $Q_1$ or $Q_3$. Therefore, the function that describes how $x$ depends on $\{l, r, t, b\}$ can be written  as follows: 
$$f(x_r,x_b,x_l,x_t)=x_rx_b\vee x_lx_t,$$
where $x_r,x_b,x_l,x_t$ are Boolean variables corresponding to points $r,b,l,t$, respectively. In other words, a vertex $y\not \in \{x,l, r, t, b\}$ 
is adjacent to $x$ if and only if $$f(A(y,r),A(y,b),A(y,l),A(y,t))=1.$$ 

\medskip
{\it Step} 2:	Let us relax the simultaneous middle point condition to the following one: 
amongst every 5 vertically (respectively horizontally) consecutive points, call the middle three {\em weak horizontal} (respectively {\em vertical}) {\em middle points}. 
Note that if the number of points is divisible by 5, at least $\frac{3}{5}$ of them are weak vertical and at least $\frac{3}{5}$ of them are weak horizontal middle points. Using this observation
it is not hard to deduce that if there are at least 13 points, then more than half of them are weak vertical and more than half of them are weak horizontal middle points.
Therefore, there must exist a point $x$ that is simultaneously both. We can deal with this case only, as the functionality of any graph on at most 12 vertices is at most 6. 
If $x$ is simultaneously a weak vertical and weak horizontal middle point, then there must exist quintuples $l$, $x$, $m_1$, $m_2$, $r$ and $t$, $x$, $m_3$, $m_4$, $b$ (not necessarily in that order), 
where $x$ is a simultaneous weak middle point in both directions, while $m_1$, $m_2$, $m_3$ and $m_4$ are the other weak middle points in their respective quintuples. 
By removing $m_1$, $m_2$, $m_3$ and $m_4$ from the graph, we find ourselves in the configuration of Step 1 and conclude that $x$ is a function of  $\{l, r, t, b\}$ in the reduced graph.
Therefore, in the original graph $x$ is a function of $\{l, r, t, b, m_1, m_2, m_3, m_4\}$, concluding the proof.
\end{proof}

\subsubsection{Symmetric difference in permutation graphs}
\label{sec:perm-sd}

Given a graph $G$ and a pair of vertices $x,y$ in $G$, let us denote $sd(x,y)=|N(x)\otimes N(y)|$ and 
$$
{\rm sd}(G)=\max\limits_H\min\limits_{x,y\in V(H)}sd(x,y),$$ where the maximum is taken over all induced subgraphs $H$ of $G$.
With some abuse of terminology we call  ${\rm sd}(G)$ the {\it symmetric difference} of $G$. 

This parameter was used implicitly in Theorems~\ref{cwd} and \ref{thm:unit-interval} to prove results about bounded 
functionality, because by bounding the symmetric difference we bound the functionality of a graph, which is easy to see. On the other hand, from the proof of Theorem~\ref{cwd}
it follows that graphs of bounded symmetric difference extend graphs of bounded clique-width. Therefore, by showing that this parameter is unbounded in a class $X$ of graphs we 
prove that the clique-width is unbounded in $X$. Our next result shows the symmetric difference is unbounded for permutation graphs. 

\begin{theorem}\label{sdpermutation}
	For any $t \in \mathbb N$, there is a permutation graph $G$ with ${\rm sd}(G) \geq t$.
\end{theorem}

\begin{proof}

Similarly to the previous section, we make use of the geometric representation of permutations. 
Given two vertices $x$ and $y$ of a permutation graph $G$, the symmetric difference of their neighbourhoods can be represented geometrically as an area in the plane (see Figure \ref{geometricsd}). 
More precisely, a vertex different from $x$ and $y$ lies in the symmetric difference of their neighbourhoods if and only if the corresponding point of the permutation lies in the shaded area. 

\begin{figure}[ht]
	\begin{center}
		\begin{tikzpicture}[scale=0.6, transform shape]
		\filldraw 
		(2,1) circle (2pt) node[below left] {$x$}
		(5,5) circle (2pt) node[below left] {$y$};
		
		\draw[dotted] (2, 0) -- (2, 7);
		\draw[dotted] (5, 0) -- (5, 7);
		\draw[dotted] (0, 1) -- (7, 1);
		\draw[dotted] (0, 5) -- (7, 5);				
		
		\fill[gray, fill opacity = 0.2] (5,0) -- (5,1) -- (2,1) -- (2,0) -- cycle;
		\fill[gray, fill opacity = 0.2] (0,1) -- (2,1) -- (2,5) -- (0,5) -- cycle;
		\fill[gray, fill opacity = 0.2] (2,7) -- (5,7) -- (5,5) -- (2,5) -- cycle;
		\fill[gray, fill opacity = 0.2] (7,5) -- (5,5) -- (5,1) -- (7,1) -- cycle;		
				
		\end{tikzpicture}
	\end{center}
	\caption{Geometric symmetric difference of two points $x$ and $y$}
	\label{geometricsd}
\end{figure}
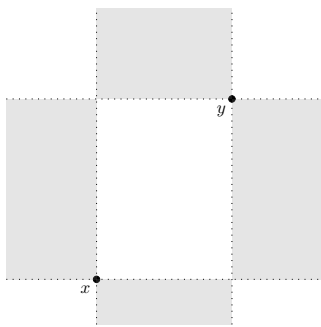

In order to prove the theorem, it suffices, for each $t \in \mathbb N$, to exhibit a set $S_t$ of points in the plane 
(with no two on the same vertical or horizontal line) such that for any pair $x, y \in S_t$, there are at least $t$ other points of $S_t$ lying in the geometric symmetric difference of $x$ and $y$. 
Such a construction immediately gives rise to a permutation and thus to a permutation graph with symmetric difference at least $t$.

We construct sets $S_t$ in the following way (see Figure \ref{exampleset} for an example): 
\begin{itemize}
	\item start with all the points with integer coordinates between 0 and $t$ inclusive;
	\item apply to the set the rotation sending $(1, 0)$ to $(1, \frac{1}{t+1})$ and $(0, 1)$ to $(-\frac{1}{t+1}, 1)$.
\end{itemize}

%
%

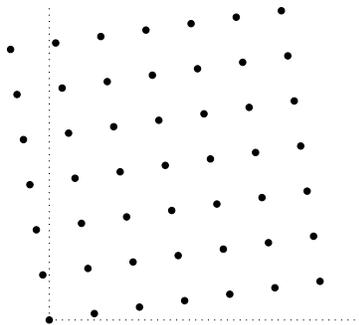
\begin{figure}[ht]
	\begin{center}
		\begin{tikzpicture}[scale=0.6, transform shape]
		
		\foreach \i in {0,...,6}
		{
			\foreach \x in {0,...,6}
			{
			\filldraw (\i - \x/7, \i/7 + \x) circle (2pt) node{}; 
			}
		}
		\draw[dotted] (0, 0) -- (0, 7);
		\draw[dotted] (0, 0) -- (7, 0);
		
		\end{tikzpicture}
	\end{center}
	\caption{The set $S_6$}
	\label{exampleset}
\end{figure}

To see that these sets have indeed the desired property, let $x, y \in S_t$. For simplicity, we will use the coordinates of the points before the rotation. Suppose $x = (x_1, x_2)$ and $y = (y_1, y_2)$. There are two possible cases (after switching $x$ and $y$ if necessary):

\begin{itemize}
	\item If $x_1 \leq y_1$ and $x_2 \leq y_2$, the following points all lie in the symmetric difference of $x$ and $y$:
	\begin{itemize}
		\item[(1)] Points $(x_1, k)$ with $k < x_2$ (in the bottom region).
		\item[(2)] Points $(x_1, k)$ with $x_2 < k \leq y_2$ (in the left region).
		\item[(3)] Points $(y_1, k)$ with $y_2 < k$ (in the top region).
		\item[(4)] Points $(y_1, k)$ with $x_2 \leq k < y_2$ (in the right region).
	\end{itemize}
	In particular, (1) and (3) account for at least $x_2 + t - y_2$ points, while (2) and (4) account for $2(y_2 - x_2)$ others. We conclude that in total, at least $t + (y_2 - x_2) \geq t$ points lie in the symmetric difference of $x$ and $y$. 
	
	\item If $x_1 \leq y_1$ and $x_2 > y_2$, the following points all lie in the symmetric difference of $x$ and $y$:
	\begin{itemize}
		\item[(1)] Points $(k, y_2)$ with $x_1 \leq k < y_1$ (in the bottom region).
		\item[(2)] Points $(k, x_2)$ with $k < x_1$ (in the left region).
		\item[(3)] Points $(k, x_2)$ with $x_1 < k \leq y_1$ (in the top region).
		\item[(4)] Points $(k, y_2)$ with $y_1 < k$ (in the right region).
	\end{itemize}
	Summing up, we find again at least $t$ points in the symmetric difference of $x$ and $y$. 
\end{itemize}
\end{proof}

This result together with Theorem~\ref{cwd} give an alternative proof of the following known fact.

\begin{cor}
	The class of permutation graphs has unbounded clique-width.
\end{cor}

\subsection{Line graphs and generalization}

The line graph $L(G)$ of a graph $G$ is the graph with vertex set $E(G)$ in which two vertices are adjacent if and only if 
the corresponding edges of $G$ share a vertex. In other words, $L(G)$ is the intersection graph of edges of $G$.  Both 
clique-width \cite{four} and degeneracy is unbounded in the class of line graphs. In this section we show that the functionality of a line graph is at most 6. 

\begin{theorem} 
The functionality of line graphs is at most 6.
\end{theorem}

\begin{proof}
Let $G$ be a graph and $H$ be the line graph of $G$. Since the class of line graphs is hereditary, it suffices  to prove that $H$ has a vertex of functionality at most 6.
We will prove a stronger result showing that {\it every} vertex of $H$ has functionality at most $6$. 

Let $x$ be a vertex in $H$, i.e. an edge in $G$. We denote the two endpoints of this edge in $G$ by $a$ and $b$. 
Assume first that both the degree of $a$ and the degree of $b$ is at least 4. 
Let $Y=\{y_1,y_2,y_3\}$ be a set of any three edges of $G$ incident to $a$, and 
let $Z=\{z_1,z_2,z_3\}$ be a set of any three edges of $G$ incident to $b$.

We claim that a vertex $v\not \in \{x\}\cup Y\cup Z$ is adjacent to $x$ in $H$ if and only if it is adjacent to every vertex in $Y$ or to every vertex in $Z$. Indeed, if 
$v$ is adjacent to $x$ in $H$, then the edge $v$ intersects the edge $x$ in $G$. If the intersection consists of $a$, then $v$  is adjacent to every vertex in $Y$ in the graph $H$, and if
the intersection consists of $b$, then $v$  is adjacent to every vertex in $Z$ in the graph $H$. Conversely, let $v$ be adjacent to every vertex in $Y$, then $v$ must intersect 
the edges $y_1,y_2,y_3$ in $G$ at vertex $a$, in which case $v$ is adjacent to $x$  in $H$. Similarly, if $v$ is adjacent to   every vertex in $Z$, then $v$ intersects 
the edges $z_1,z_2,z_3$ in $G$ at vertex $b$ and hence $v$ is adjacent to $x$ in $H$.

Therefore, in the case when both $a$ and $b$ have degree at least 4 in $G$, 
the function that describes how $x$ depends on $\{y_1,y_2,y_3,z_1,z_2,z_3\}$ in the graph $H$ can be written  as follows: 
$$f(y_1,y_2,y_3,z_1,z_2,z_3)=y_1y_2y_3\vee z_1z_2z_3.$$

If the degree of $a$ is less than 4, we include in $Y$ all the edges of $G$ distinct from $x$ which are incident to $a$ (if there are any) and remove the term $y_1y_2y_3$ from the function.
Similarly, if  the degree of $b$ is less than 4, we include in $Z$ all the edges of $G$ distinct from $x$ which are  incident to $b$ (if there are any) and remove the term $z_1z_2z_3$ from the function.
If both terms have been removed, the function is defined to be identically 0, i.e. no vertices are adjacent to $x$ in $H$, except for those in $Y\cup Z$.   
\end{proof}

Having proved that the intersection graph of edges, i.e. the intersection graph  of a family of $2$-subsets, has bounded functionality, it is  natural to ask whether the intersection graph of a family of $k$-subsets
has bounded functionality for $k>2$. This question is substantially harder and we present a solution only for $k=3$.  

\subsubsection{Line graphs of 3-uniform hypergraphs}

In this section we will show that intersection graphs of 3-uniform hypergraphs is a class of bounded functionality. 
We will denote a 3-uniform hypergraph with the ground set $V$ by $(V, \mathcal{S})$, where $\mathcal{S} \subseteq V \times V \times V$. 
We will use variables $s, s', s_1, s_2, \ldots$ to denote hyperedges, i.e. the elements of $\mathcal{S}$, and variables $v, v_1, v_2,\ldots $ to denote the elements of $V$. 
We will say that two hyperedges $s$ and $s'$ intersect if $s \cap s' \neq \emptyset$. We start with a preparatory result.  

\begin{lemma}\label{lemma1}
Let $(V, \mathcal{S})$ be a 3-uniform hypergraph and $v \in V$. Then one of the following holds:
\begin{itemize}
\item There are 3 hyperedges $s_1, s_2, s_3$ such that $s_i \cap s_j = \{v\}$ for all $1\leq i<j \leq 3$. 
\item There are 4 vertices $v_1, v_2, v_3, v_4$ such that each hyperedge $s \in \mathcal{S}$ that contains $v$ also contains at least one of the $v_1, v_2, v_3$ or $v_4$. 
\end{itemize} 
\end{lemma}

\begin{proof}
Consider the set $\mathcal{E}=\{s \backslash \{v\} : s \in \mathcal{S}, v \in s\}$. This is the set of pairs of vertices that are obtained by removing vertex $v$ from the hyperedges that contain $v$. 
Therefore, $(V, \mathcal{E})$ can be viewed as a graph. The lemma now says that either this graph contains a matching with 3 edges (as a subgraph) or it contains 4 vertices that any edge is adjacent to (vertex cover of size 4).
The proof of this is now easy. One can take a maximal matching $M$, and if it has at least 3 edges, then we are done. 
In the other case, when the maximal matching $M$ has at most two edges, take $v_i$'s to be the vertices of the matching.
If needed, add arbitrary vertices to obtain a set of 4 vertices. 
By maximality of the matching, every hyperedge contains at least one of the vertices selected, hence we are done as well.  
\end{proof}

The following two easy observations will be needed in the course of the proof.

\begin{obs} \label{obs1}
Let $(V, \mathcal{S})$ be a 3-uniform hypergraph. Suppose hyperedges $s_1, s_2, s_3 \in \mathcal{S}$ pairwise intersect at exactly one vertex, say $\{v\} = s_1 \cap s_2 = s_2 \cap s_3 = s_3 \cap s_1$.  
In other words, $s_1=\{v, v_1, v_2\}$, $s_2=\{v, v_3, v_4\}$, $s_1=\{v, v_5, v_6\}$, for some distinct vertices $v_1, v_2, \ldots, v_6$. 
Let $F'=\{(v_i, v_j, v_k): 1 \leq i \leq 2,  3 \leq j \leq 4, 5 \leq k \leq 6\}$ be the set of 8 hyperedges that intersect each of $s_1, s_2, s_3$ in exactly one vertex that is different from $v$. 
Then one can easily determine whether a given edge $s' \in \mathcal{S} \backslash F'$ contains vertex $v$ or not by looking at the intersection of $s'$  with $s_1, s_2, s_3$. 
Indeed, $s'$ contains $v$ if and only if $s'$ intersects each of $s_1$, $s_2$ and $s_3$.
\end{obs}

\begin{obs} \label{obs2}
Let $(V, \mathcal{S})$ be a 3-uniform hypergraph. Suppose hyperedges $s_1, s_2, s_3 \in \mathcal{S}$ pairwise intersect at exactly 2 vertices. 
In other words, $s_1=\{v_1, v_2, v_3\}$, $s_2=\{v_1, v_2, v_4\}$ and $s_3=\{v_1, v_2, v_5\}$, for some distinct vertices $v_1, v_2, v_3, v_4, v_5 \in V$. 
Let $F'$ be the set containing the hyperedge $\{v_3, v_4, v_5\}$. Then one can easily determine whether a given edge $s' \in \mathcal{S} \backslash F'$ contains at least one of the vertices $v_1, v_2$ or not 
by looking at the intersection of $s'$  with $s_1, s_2, s_3$. Indeed, $s'$ contains $v_1$, $v_2$ or both if and only if $s'$ intersects each of $s_1$, $s_2$ and $s_3$.
\end{obs}

\begin{definition}
Let $(V, \mathcal{S})$ be a 3-uniform hypergraph and let $v_1, v_2 \in V$. We will call the pair $v_1v_2$ {\it thick} if there are at least 32 hyperedges in $\mathcal{S}$ that contain $\{v_1, v_2\}$. 
\end{definition}

We will split our analysis into two cases. In the first lemma we will show that the intersection graphs of 3-uniform hypergraphs without thick pairs have bounded functionality. 
In the second lemma we will provide a structural theorem about hypergraphs containing thick pairs, from which a bounded functionality result follows easily as well. 
We note that in the case without thick pairs, we provide a bound on functionality for {\it  any} vertex of the intersection graph.
Meanwhile, in the case of hypergraphs with thick pairs, for any given bound $M$ one can find a hypergraph and a hyperedge such that corresponding vertex in the intersection graph has functionality at least $M$. 
Thus a structural result is needed in this case, to show that we can find a {\it particular} hyperedge in any given hypergraph with thick pairs, 
such that the functionality of the vertex corresponding to the hyperedge is bounded by a constant, that does not depend on the hypergraph.

We start with the case when there are no thick pairs.

\begin{lemma} \label{lemmawithoutthick}
Let $(V, \mathcal{S})$ be a 3-uniform hypergraph without thick pairs. Then for any hyperedge $s \in \mathcal{S}$ there is a set of hyperedges $F \subset \mathcal{S} \backslash \{s\}$ of size $|F| \leq 462$ 
such that for any $s' \in \mathcal{S}\backslash (F \cup \{s\})$, one can determine whether $s'$ intersects $s$ by looking at the intersections of $s'$ with the hyperedges of $F$.    
\end{lemma}

\begin{proof}

Let $s$ be any hyperedge in the hypergraph. Since we assume that there are no thick pairs, there are at most $30 \times 3$ hyperedges in $(V,\mathcal{S})$ that intersect $s$ in exactly 2 vertices.
We denote this set of at most 90 hyperedges by $F_1$. Let $v \in s$, and consider the hyperedges in $(V, \mathcal{S} \backslash (F_1 \cup \{s\}))$ that contain vertex $v$. 
By Lemma~\ref{lemma1} we can distinguish between the following two cases.
\begin{itemize}
\item Assume there exist 3 hyperedges $s_1, s_2, s_3$ in $(V, \mathcal{S} \backslash (F_1 \cup \{s\}))$ that pairwise intersect at vertex $v$ only.
In this case, we denote by $F_2$ the set of at most 11 hyperedges consisting of $s_1, s_2, s_3$ and all the hyperedges in $\mathcal{S}$ that have exactly one vertex in each of 
$s_1 \backslash \{v\}$, $s_2 \backslash \{v\}$ and $s_3 \backslash \{v\}$. According to Observation~\ref{obs1} we can determine whether a given hyperedge $s'\in \mathcal{S} \backslash (F_1 \cup F_2 \cup \{s\})$ 
contains $v$ or not by looking at the intersection of $s'$  with $s_1, s_2, s_3$. 
\item Suppose now that there exists a set of vertices $\{v_1, v_2, v_3, v_4\}$ such that every 
hyperedge in $(V, \mathcal{S} \backslash (F_1 \cup \{s\}))$ that contains $v$ also contains at least one of $v_1, v_2, v_3$ or $v_4$. 
In this case, we denote by $F_2$ the set of all the hyperedges that contain at least one of the pairs $\{v, v_1\}$, $\{v, v_2\}$, $\{v, v_3\}$ or $\{v, v_4\}$. 
By our assumption on no thick pairs, the set $F_2$ contains at most $31\times 4=124$ edges. Observe that no hyperedge $s' \in \mathcal{S} \backslash (F_1 \cup F_2 \cup \{s\})$ intersects $v$. 
\end{itemize}
By analogy with building the set $F_2$ for the vertex $v$, we build two more 
sets $F_3$ and $F_4$ for the other two vertices contained in the hyperedge $s$, i.e. for the  vertices in $s\backslash \{v\}$. 
Now it is easy to see that the $F=F_1 \cup F_2 \cup F_3 \cup F_4$ allows us to determine whether a given 
hyperedge $s' \in S \backslash (F \cup \{s\})$ intersects $s$ or not. Note that $F$ has size at most $90+3 \times 124=462$.       
\end{proof}

In our next result, we will show that a 3-uniform hypergraph with a thick pair contains one of the structures presented in Figure~\ref{thickedge}, which we call ``fly'', ``windmill'', and ``broken windmill''.

\begin{figure}[h]
    \centering
    \begin{subfigure}[b]{0.3\textwidth}
    \centering
    \begin{tikzpicture}[scale=.6,auto=left]
		\foreach \start in {0}
		{		
		\node[w_vertex] (1) at (\start+0, 2) { }; 	
		\node[w_vertex] (2) at (\start-1, 0) { };
		\node[w_vertex] (3) at (\start+1, 0) { };

		\node[w_vertex] (4) at (\start+1, 1) { };
		\node[w_vertex] (5) at (\start+2,1) { };
		\node[w_vertex] (6) at (\start+3,1) { };

		\node[w_vertex] (7) at (\start-1,1) { };
		\node[w_vertex] (8) at (\start-2,1) { };
		\node[w_vertex] (9) at (\start-3,1) { };

		\draw (\start+0, 2) node[anchor=south] {$v_{1}$};
		\draw (\start-1, 0) node[anchor=north] {$v_{2}$};
		\draw (\start+1, 0) node[anchor=north] {$v_{3}$};

		\foreach \from/\to in {1/2,2/3,3/1,1/4,1/5,1/6,1/7,1/8, 1/9, 3/4, 3/5, 3/6, 2/7, 2/8, 2/9}
	    	\draw (\from) -- (\to);
		}
          \end{tikzpicture}
    
       \vspace{3mm}

        \caption{A ``fly''}
       
    \end{subfigure}
    ~
    \begin{subfigure}[b]{0.3\textwidth}
    \centering
    \begin{tikzpicture}[scale=.6,auto=left]
		\foreach \start in {0}
		{		
		\node[w_vertex] (1) at (\start+0, 2) { }; 
		\node[w_vertex] (2) at (\start-1, 0) { };
		\node[w_vertex] (3) at (\start+1, 0) { };

               \node[w_vertex] (4) at (\start+0, -0.4) { };
		\node[w_vertex] (5) at (\start+0, -0.8) { };
		\node[w_vertex] (6) at (\start+0, -1.2) { };

                \node[w_vertex] (7) at (\start-2, 2.2) { };
		\node[w_vertex] (8) at (\start-2, 3) { };

		\node[w_vertex] (9) at (\start+2, 2.2) { };
                \node[w_vertex] (10) at (\start+2, 3) { };

		\node[w_vertex] (11) at (\start-0.5, 4) { };
		\node[w_vertex] (12) at (\start+0.5, 4) { };

		\draw (\start+0.4, 2) node[anchor=south] {$v_{1}$};
		\draw (\start-1.2, 0) node[anchor=north] {$v_{2}$};
		\draw (\start+1.2, 0) node[anchor=north] {$v_{3}$};

		\foreach \from/\to in {1/2,2/3,3/1,2/4,2/5,2/6, 3/4, 3/5, 3/6, 1/7, 1/8, 1/9, 1/10, 1/11, 1/12, 7/8, 9/10, 11/12}
	    	\draw (\from) -- (\to);		}
          \end{tikzpicture}

        \caption{A ``windmill''}
        
    \end{subfigure}
    ~ 
    \begin{subfigure}[b]{0.3\textwidth}      
  \centering
  \qquad
  \begin{tikzpicture}[scale=.6,auto=left]
		\foreach \start in {0}
		{		
		\node[w_vertex] (1) at (\start+0, 2) { }; 	
		\node[w_vertex] (2) at (\start-1, 0) { };
		\node[w_vertex] (3) at (\start+1, 0) { };

		\node[w_vertex] (4) at (\start+0, -0.4) { };
		\node[w_vertex] (5) at (\start+0, -0.8) { };
		\node[w_vertex] (6) at (\start+0, -1.2) { };

		\draw (\start+1.3, 2.1) node[anchor=south] {$v_{1}$ {\tiny bounded degree}};
		\draw (\start-1.2, 0) node[anchor=north] {$v_{2}$};
		\draw (\start+1.2, 0) node[anchor=north] {$v_{3}$};

		\foreach \from/\to in {1/2,2/3,3/1,2/4,2/5,2/6, 3/4, 3/5, 3/6}
	    	\draw (\from) -- (\to);

                \draw(1) -- (\start-0.4, 2.2);
		\draw(1) -- (\start-0.3, 2.3);
		\draw(1) -- (\start-0.1, 2.3);
 		\draw(1) -- (\start+0.1, 2.3);
                \draw(1) -- (\start+0.3, 2.3);
		\draw(1) -- (\start+0.4, 2.2);

		}
          \end{tikzpicture}

        \caption{A ``broken windmill''}
       
    \end{subfigure}
    \caption{Substructures that appear in a 3-uniform hypergraph with a thick pair}\label{thickedge}
\end{figure}
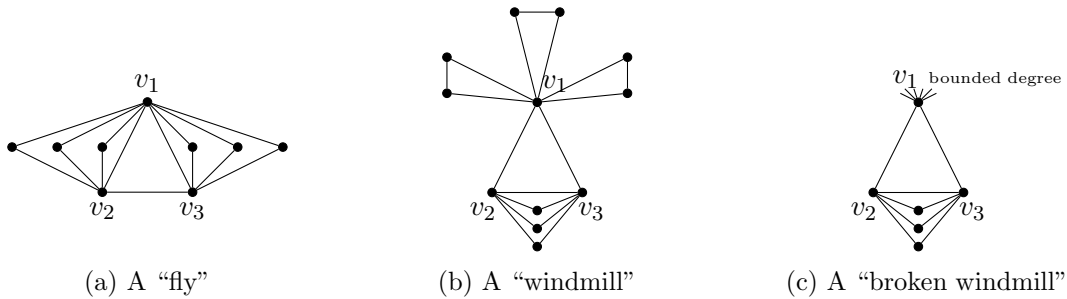

To prove the result about these three structures, we need the following observation. 

\begin{obs}\label{obs3}
Let $(V,\mathcal{S})$ be a 3-uniform hypergraph and let $v \in V$ be a vertex that does not belong to any thick pair. Then one of the following holds:
\begin{itemize}
\item Either there are 3 hyperedges $s_1, s_2, s_3$ that pairwise intersect only at vertex $v$.
\item Or vertex $v$ is contained in at most 124 hyperedges of $(V,\mathcal{S})$. 
\end{itemize}
\end{obs}

\begin{proof}
From Lemma~\ref{lemma1}, it follows that either there are 3 hyperedges $s_1, s_2, s_3$ that pairwise intersect only at vertex $v$,   
or there are 4 vertices $v_1, v_2, v_3$ and $v_4$ such that every hyperedge that contains $v$ also contains at least one of $v_1, v_2, v_3$ or $v_4$.
Note that in the second case, since neither of $vv_1, vv_2, vv_3$ and $vv_4$ is thick, there are at most 31 hyperedges containing one of these pairs. 
Therefore, there are at most $31 \times 4 = 124$ hyperedges that contain $v$. This finishes the proof of the observation. 
\end{proof}

\begin{lemma}\label{lemmawiththick}
Let $(V, \mathcal{S})$ be a 3-uniform hypergraph that contains a thick pair. Then it contains one of the following:

\begin{itemize}

\item A ``fly'', which is a hyperedge $s=\{v_1, v_2, v_3\}$ together with hyperedges $s_1, s_2, s_3, s_4, s_5, s_6$ such that $s_1, s_2, s_3$ intersect $s$ at $\{v_1,v_2\}$ and $s_4, s_5, s_6$ intersect $s$ at $\{v_1, v_3\}$. 

\item A ``windmill'', which is a hyperedge $s=\{v_1, v_2, v_3\}$ together with hyperedges $s_1, s_2, s_3, s_4, s_5$, $s_6$ such that 
$s_1, s_2, s_3$ intersect $s$ at $\{v_2, v_3\}$ and $s_4, s_5, s_6$ intersect $s$ at $v_1$ and the pairwise intersection of $s_4, s_5, s_6$ is vertex $v_1$ as well.

\item A ``broken windmill'', which is a hyperedge $s=\{v_1, v_2, v_3\}$ together with hyperedges $s_1, s_2, s_3$ such that 
$s_1, s_2, s_3$ intersect $s$ at $\{v_2, v_3\}$ and there are only at most $124$ hyperedges in $(V, \mathcal{S}\backslash \{s\})$ that contain vertex $v_1$.
\end{itemize}
\end{lemma}

\begin{proof}
Let $T=\{(v, v_1v_2): \{v, v_1, v_2\} \in \mathcal{S}, v_1v_2$ is a thick pair$\}$. The main idea of the proof is counting the elements in $T$.
Let $E$ be the set of thick pairs and let $W=\{v \in V :  (v, v_1v_2) \in T$ for some $ v_1, v_2 \in V\}$. 

Note that any thick pair belongs to at least 32 hyperedges, hence $|T| \geq 32 |E|$. 
Also note that each vertex of $W$ belongs to some thick pair, or else a ``windmill'' or a ``broken windmill'' appears. 
Indeed, assume $s=\{v, v_1, v_2\}$ is a hyperedge with a thick pair $v_1v_2$ and $v$ does not belong to any thick pair. 
Then, by Observation~\ref{obs3}, either there are at most 124 hyperedges containing vertex $v$ in $(V, \mathcal{S} \backslash \{s\})$
or there are three hyperedges $s_1$, $s_2$, $s_3$ in $(V, \mathcal{S} \backslash \{s\})$ that pairwise intersect only at vertex $v$. 
Together with any three hyperedges different from $s$ that contain the vertices of the thick pair $\{v_1,v_2\}$ and are different from $s$, this gives us either a ``windmill'' or a ``broken windmill''. 

Thus, from now on, we will assume that each vertex of $W$ belongs to a thick pair. As each thick pair contains
at most two vertices, we have $|W| \leq 2 |E|$. We conclude that $|T| \geq 32 |E| \geq 16 |W|$. Hence, by the pigeonhole principle, there are 16 elements of 
$T$ that contain the same element $v \in W$. In other words, we have 16 different thick pairs $e_1, e_2, \ldots, e_{16}$ that make a hyperedge with vertex $v$, i.e. $(v, e_i) \in T$ for all $i$. 

Consider first the case when four of these pairs, say $e_1, e_2, e_3, e_4$, have a vertex $w$ in common. 
Denote these pairs by $e_1=wz$, $e_2=wu_1$, $e_3=wu_2$ and $e_4=wu_3$. Then $s=\{v,w,z\}$, with $s_1=\{v,w,u_1\}$, $s_2=\{v,w,u_2\}$, $s_3=\{v,w,u_3\}$, 
together with any three hyperedges that contain the thick pair $e_1=wz$ and are different from $s$, gives us a ``fly''. 

Finally, consider the case when no four pairs of $e_1, e_2, \ldots, e_{16}$ share a vertex in common. Then a graph $G$ with the edges set $\{e_1, e_2, \ldots, e_{16}\}$
has degree at most 3 and hence it must contain  a matching of size at least 4.  Indeed, by observing that each edge of $G$ is incident to at most 4 other edges, 
one can pick any edge and remove all incident edges repeatedly at least 4 times to obtain the required matching of size 4.  
Now, as each of these four edges is a thick pair and forms a hyperedge with $v$, we can easily see that a ``windmill'' appears. This finishes the proof.
\end{proof}

\begin{cor} \label{cor1}
Let $(V, \mathcal{S})$ be a 3-uniform hypergraph that contains a thick pair. Then there is a hyperedge $s \in \mathcal{S}$ and a set of hyperedges 
$F \subset \mathcal{S} \backslash \{s\}$ of size $|F| \leq 128$ such that such that 
for any $s' \in \mathcal{S} \backslash (F \cup \{s\})$, one can determine whether $s'$ intersects $s$ by looking at the intersections of  $s'$ with the hyperedges of $F$.    
\end{cor}

\begin{proof}
We use the notation of the statement of Lemma~\ref{lemmawiththick}. Let $s$ be a hyperedge given by Lemma~\ref{lemmawiththick} belonging either to a ``fly'' or to a ``windmill'' or to a ``broken windmill''.

If $s$ belongs to a ``fly'', then we can take $F$ to consist of 6 hyperedges $s_1, s_2, \ldots, s_6$ together with 2 further possible hyperedges on the wings of the fly 
(if the hypergraph contains it) on vertices $(s_1 \cup s_2 \cup s_3) \backslash \{v_1,v_2\}$ and $(s_4 \cup s_5 \cup s_6) \backslash \{v_1,v_3\}$. 
It now follows from Observation~\ref{obs2} that intersecting any hyperedge $s' \in S \backslash (F \cup s)$, with $s_1, s_2$ and $s_3$ one can determine whether $s'$ contains either $v_1$ or $v_2$. 
Similarly, intersecting with $s_4, s_5, s_6$, determines whether $s'$ contains either $v_1$ or $v_3$. Hence, by looking at the intersection of the edges of $F$ with $s'$ we can determine whether $s'$ intersects $s$ or not.

If $s$ belongs to a ``windmill'', we can take $F$ to consist of 6 hyperedges $s_1, s_2, \ldots, s_6$ together with one possible hyperedge on $(s_1 \cup s_2 \cup s_3) \backslash \{v_2, v_3\}$ and 
8 possible hyperedges on $(s_4 \cup s_5 \cup s_6) \backslash \{v_1\}$ that have one vertex in each wing of the ``windmill''. By Observation~\ref{obs2} intersection of $s'$ with $s_1,s_2, s_3$ determines whether 
$s'$ contains either $v_2$ or $v_3$, while Observation~\ref{obs1} allows us to determine whether $s'$ contains $v_1$ by looking at the intersection of $s'$ with $s_4, s_5$ and $s_6$. 
Thus, again we can determine whether $s'$ intersects $s$. 

If $s$ belongs to a ``broken windmill'', we can take $F$ to consist of all the hyperedges that contain vertex $v_1$, of  which there are at most 124, 
also with $s_1, s_2, s_3$ and one further possible hyperedge on the set  $(s_1 \cup s_2 \cup s_3) \backslash \{v_2, v_3\}$. 
Now $F$ contains at most 128 edges and it is clear by Observation~\ref{obs2} that this set determines whether $s'$ intersects $s$ or not.
\end{proof}

From Lemma~\ref{lemmawithoutthick} and Corollary~\ref{cor1} we deduce our main result of this section.
\begin{theorem}
Intersection graphs of 3-uniform hypergraphs have functionality bounded by 462.
\end{theorem}


\section{Graphs of large functionality}
\label{sec:unbounded}
Knowing what is good without knowing what is bad is just half-knowledge. Therefore, in this section we turn to graphs of large functionality.

When we talk about graphs of large functionality we assume that we deal with an infinite family $X$ of graphs, because in any finite collection of graphs functionality is bounded by a constant. 
Moreover, we can further assume that $X$ is hereditary. Indeed, if $X$ is not hereditary, we can extend it to a hereditary class by adding 
all induced subgraphs of graphs in $X$, and this extension has (un)bounded functionality if and only if $X$ has, because by definition the functionality of an induced subgraph of a  graph $G$ is never 
larger than the functionality of $G$. 

In \cite{implicit}, it was shown that any hereditary class of graphs of bounded functionality has $2^{O(n\log_2 n)}$ labelled graphs with $n$ vertices. 
In the terminology of \cite{SpHerProp} these are classes with (at most) factorial speed of growth, or simply (at most) factorial classes. 
Therefore, in every superfactorial class functionality is unbounded. This is the case, for instance, for bipartite, co-bipartite and split graphs, 
since each of these classes contains at least $2^{n^2/4}$ labelled graphs with $n$ vertices. We state this formally as a lemma.

\begin{lemma}\label{lem:three-classes}
Functionality is unbounded in the  classes of bipartite, co-bipartite and split graphs.  
\end{lemma}

This conclusion allows us to establish a relationship between functionality and one more important graph parameter known as VC-dimension. 

A set system $(X,S)$ consists of a set $X$ and a family $S$ of subsets of $X$. 
A subset $A\subseteq X$ is {\it shattered} if for every subset $B\subseteq A$
there is a set $C\in S$ such that $B=A\cap C$. The VC-dimension of $(X,S)$
is the cardinality of a largest shattered subset of $X$.

The VC-dimension of a graph $G=(V,E)$ was defined in \cite{VC} as the VC-dimension of 
the set system $(V,S)$, where $S$ the family of closed neighbourhoods of vertices of $G$,
i.e. $S=\{N[v]\ :\ v\in V(G)\}$. We denote the VC-dimension of $G$ by $vc(G)$.

\begin{theorem}
There exists a function $f$ such that for any graph $G$, $vc(G)\le f(fun(G))$.
\end{theorem}

\begin{proof}
Fix a $k$ and consider the class $X_k$ of all graphs of functionality at most $k$. Clearly, $X_k$ is hereditary.
Assume $X_k$ contains graphs of arbitrarily large VC-dimension and let $G_1,G_2,\ldots$ be an infinite sequence of graphs from $X_k$ with strictly increasing values of the VC-dimension. 
Let $Y$ be the hereditary class containing all these graphs and all their induced subgraphs. Then $Y$ is a hereditary subclass of $X_k$ with unbounded VC-dimension. 
It is was shown in \cite{Lozin} that the only minimal hereditary classes of graph of unbounded VC-dimension are bipartite, co-bipartite and split graphs. But then $Y$ and hence $X_k$ 
contains one of these three classes, which is a contradiction to Lemma~\ref{lem:three-classes}.
Therefore, there is a constant $f(k)$ bounding the VC-dimension of graphs in $X_k$, which defines the function $f$.
\end{proof}

This theorem shows that the family of classes of bounded VC-dimension extends the family of classes of bounded functionality. 
The next result shows that this extension is proper and reveals several classes of unbounded functionality and bounded VC-dimension. 

\begin{theorem}\label{thm:five-classes}
The following classes have bounded VC-dimension and unbounded functionality:
\begin{itemize}
\item chordal bipartite graphs,
\item complements of chordal bipartite graphs,
\item strongly chordal graphs,
\item bipartite graphs of girth at least $k$, for a fixed value of $k$,
\item finitely defined monotone classes containing all forests.
\end{itemize}
\end{theorem}

\begin{proof}
To see that VC-dimension is bounded in all these classes, observe that none of them contains any of the three minimal classes of unbounded VC-dimension 
(bipartite, co-bipartite and split graphs).

To prove the unboundedness of functionality in these classes, we will show that all of them are superfactorial. 
For the first three classes a superfactorial bound on the number of $n$-vertex labelled graphs was shown in \cite{Spinrad} and it equals $2^{\Theta(n\log^2_2 n)}$.

Now we turn to the last two families of graph classes and observe that all of them are monotone. 
It is known (see e.g. \cite{girth}) that for each $k$ there exist $(C_3,C_4,\ldots,C_k)$-free bipartite graphs with $n$ vertices and $\Omega (n^{1+1/k})$ edges.
Since the class of $(C_3,C_4,\ldots,C_k)$-free bipartite graphs is monotone, we conclude that the number of $n$-vertex labelled graphs  in this class is at least $2^{\Omega (n^{1+1/k})}$,
i.e. the class is superfactorial. 

Finally, let $X$ be a finitely defined monotone class containing all forests. Since $X$ contains all forests, every forbidden graph for $X$ contains a cycle, and since the number of forbidden graphs is finite, 
there is a largest $k$ such that every forbidden graph contains an induced cycle of length at most $k$. Therefore, $X$ contains all $(C_3,C_4,\ldots,C_k)$-free graphs,
and hence, as before, $X$ contains graphs with $n$ vertices and $\Omega (n^{1+1/k})$ edges. Since $X$ is monotone, it contains at least $2^{\Omega (n^{1+1/k})}$ labelled graphs with $n$ vertices,
i.e. $X$ is superfactorial. 
\end{proof}

We observe that this theorem cannot be extended to the family of all monotone classes, as the example of forests shows. 
Obviously, this class is monotone (and contains all forests), but functionality is bounded by 1 in this class, since every forest has a vertex of degree at most 1.
Nevertheless, in conjunction with Theorem~\ref{thm:monotone-bounded} the above result provides the following dichotomy for finitely defined monotone classes. 
\begin{theorem}
A finitely defined monotone class has bounded functionality if and only if it does not contain all forests. 
\end{theorem}

So far, we have identified some classes containing graphs of large functionality. However, presenting specific constructions of graphs of large functionality is 
a task, which is not so straightforward. We solve it in the following section.

\subsection{Constructing graphs of large functionality}

Since large VC-dimension implies large functionality, it would be natural to construct graphs of large functionality through constructing graphs of large VC-dimension.
The latter is an easy task. Indeed, consider the bipartite graph $D_n=(A,B,E)$ with two parts $|A|=n$ and $|B|=2^n$. For each subset $C\subseteq A$ we create a vertex in $B$
whose neighbourhood coincide with $C$. Clearly, the VC-dimension of $D_n$ is $n$ and hence with $n$ growing the functionality of $D_n$ grows as well. 

However, this example is not very interesting in the sense that $D_n$ contains vertices of low functionality (of low degree)  
and hence graphs of large functionality are hidden in $D_n$ as proper induced subgraphs. A much more interesting task is 
constructing graphs where {\it all} vertices have large functionality. In what follows, we show that this is the case for hypercubes.

Let $V_n=\{0,1\}^n$ be the set of binary sequences of length $n$ and let $v,w\in V_n$. The Hamming distance $d(v,w)$ between $v$ and $w$
is the number of positions in which the two sequences differ. 
A {\it hypercube} $Q_n$ is the graph with vertex set $V_n=\{0,1\}^n$, in which two vertices are adjacent if and only if the Hamming distance between them equals 1.

\begin{theorem}
Functionality of the hypercube $Q_n$ is at least $(n-1)/3$. 
\end{theorem}

\begin{proof}
By symmetry, it suffices to show that the vertex $v=00 \ldots 0 \in V_n$ has functionality at least $(n-1)/3$. 
Let $v$ be a function of vertices in a set $S \subseteq V_n \backslash \{v\}$. To provide a lower bound on the size of $S$, and hence a lower bound
on the functionality of $v$, for each $i=1, 2, \ldots, n$ consider the set $S_i=\{w \in S: d(w, v)=i\}$, i.e. the set of all binary sequences in $S$ that contain exactly $i$ 1s. 
Also, consider the following set: 
$$I=\{i \in \{1,2, \dots, n\} : \exists z = z_1z_2 \ldots z_n \in S_1 \cup S_2 \cup S_3 \  \mbox{with} \  z_i=1\}.$$
Suppose $|I|\le n-2$. Then there exist two positions $i$ and $j$ such that for any sequence $z = z_1z_2 \ldots z_n \in S_1 \cup S_2 \cup S_3$, we have $z_i=0$ and $z_j=0$. Consider the following two vertices: 
\begin{itemize}
\item $u=u_1u_2\ldots u_n$ with $u_k=1$ if and only if $k=i$,
\item $w=w_1w_2 \ldots w_n$ with  $w_k=1$ if and only if $k=i$ or $k=j$.   
\end{itemize} 

We claim that $u$ and $w$ are not adjacent to any vertex $z \in S$. 
First, it is not hard to see that for any $z \in S_1 \cup S_2 \cup S_3$ we have $d(z, u) \geq 2$ and $d(z, w) \geq 2$. 
Indeed, any $z \in S_1 \cup S_2 \cup S_3$ differs from $u$ and $w$ in position $i$, i.e. $z_i=0$ and $u_i=w_i=1$, and there must exist a $k\ne i,j$ with $z_k=1$ and $u_k=w_k=0$.
Also, it is not difficult to see that $d(z, u) \geq 2$ and $d(z, w) \geq 2$ for any vertex $z \in S \backslash (S_1 \cup S_2 \cup S_3)$, because any such $z$ has at least four 1s,
while $u$ and $w$ have at most two 1s. Therefore, by definition, $u$ and $w$ are not adjacent to any vertex in $S$.   

We see that the assumption that $|I|\le n-2$ leads to the conclusion that there are two vertices $u, w \in Q_n \backslash (S \cup \{v\})$ 
which are non-adjacent to any vertex in $S$, but have different adjacencies to $v$. This contradicts the fact that $v$ is a function of the vertices in $S$. 
So, we must conclude that $I$ has size at least $n-1$. As each vertex in $S_1 \cup S_2 \cup S_3$ has at most three 1s, 
we conclude that $S_1 \cup S_2 \cup S_3$ must contain at least $|I|/3=(n-1)/3$ vertices. This completes the proof of the theorem.
\end{proof}

We conclude this section by observing that the hereditary closure of the set of hypercubes, i.e. the hereditary class containing 
all the hypercubes and all their induced subgraphs, is one more example of a hereditary class of unbounded functionality and bounded VC-dimension. 
The difference between this example and the classes in Theorem~\ref{thm:five-classes} is that the speed of the hereditary closure of hypercubes 
is an open question. We discuss some other open questions related to the topic of the paper in the concluding section.


\section{Concluding remarks and open problems}
\label{sec:conclusion}

In this paper, we proved a number of results about graph functionality. However, many questions on this topic remain unanswered.  

\subsection{Bounded functionality, implicit representation and factorial properties of graphs}

Let us repeat that any hereditary class of graphs of bounded functionality is at most factorial \cite{implicit}.
It is natural to ask whether all factorial classes are of bounded functionality. 
\begin{problem}
Is it true that for any hereditary class with at most factorial speed of growth there exists a constant bounding the functionality of graphs in the class? 
\end{problem}
To emphasize the importance of the family of factorial classes let us mention that it contains many classes of theoretical or practical importance,
such as line graphs, interval graphs, permutation graphs, threshold graphs, forests, 
planar graphs and, even more generally, all proper minor-closed graph classes, 
all classes of graphs of bounded vertex degree, of bounded clique-width, etc. 

There is one more important notion associated with factorial classes of graphs, namely, the notion of  implicit representation of graphs,
which was introduced in \cite{implicit-0} and then further developed in \cite{implicit-1}. Similarly to bounded functionality, any hereditary class that admits 
an implicit representation is at most factorial. However, the question whether all factorial classes admit implicit 
representations, also known as {\it implicit graph representation conjecture}, is widely open.
We ask whether there is any relationship between the two notions. 
\begin{problem}
Does implicit representation implies bounded functionality and/or vice versa? 
\end{problem}

\subsection{Other open questions}

We conclude the paper with a number of other open questions related to the notion of graph functionality. Some of them are motivated by the results presented in the paper.
We list them in no particular order. The first of them is inspired by a result in \cite{implicit} showing that if the family of prime (with respect to modular decomposition) graphs 
in a hereditary class $X$ is factorial, then the entire class $X$ is factorial.

\begin{problem}
Is it true that if prime (with respect to modular decomposition) graphs 
in a hereditary class $X$ have bounded functionality, then all graphs in $X$ have bounded functionality?
\end{problem}

\begin{problem}
Describe explicitly any function satisfying $vc(G)\le f(fun(G))$.  
\end{problem}

\begin{problem}
Are there any minimal hereditary classes of unbounded functionality?  
\end{problem}

\begin{problem}
Is functionality bounded for interval graphs, or more generally, for graphs of bounded boxicity?  
\end{problem}

\begin{problem}
What is the time complexity of computing the functionality of a graph?  
\end{problem}

\begin{problem}
Are there NP-hard problems that admit polynomial-time or fixed-parameter tractable algorithms for graphs of bounded functionality? 
\end{problem}

\end{document}